\documentclass[12pt,a4paper]{amsart}
\usepackage[utf8]{luainputenc}
\usepackage{array}
\usepackage{units}
\usepackage{amstext}
\usepackage{amsthm}
\usepackage{amssymb}
\usepackage{graphicx}
\usepackage[numbers]{natbib}
\usepackage[unicode=true,
 bookmarks=true,bookmarksnumbered=false,bookmarksopen=false,
 breaklinks=false,pdfborder={0 0 0},backref=false,colorlinks=false]
 {hyperref}
\hypersetup{pdftitle={Cyclotomic Aperiodic Substitution Tilings with Dense Tile Orientations},
 pdfauthor={Stefan Pautze},
 pdfsubject={The class of Cyclotomic Aperiodic Substitution Tilings (CAST) whose vertices are supported by the $2n$-th cyclotomic field $\mathbb{Q}\left(\zeta_{2n}\right)$ is extended to cases with Dense Tile Orientations (DTO). It is shown that every CAST with DTO has an inflation multiplier {\normalsize{}$\eta$} with irrational argument so that $\nicefrac{k\pi}{2n}\neq\arg\left(\eta\right)\notin\mathbb{\pi Q}$. The minimal inflation multiplier {\normalsize{}$\eta_{min.irr.}$} is discussed for $n\geqq2$. Examples of CASTs with DTO, minimal inflation multiplier{\normalsize{} $\eta_{min.irr.}$} and individual dihedral symmetry $D_{2n}$ are introduced for $n\in\left\{ 2,3,4,5,6,7\right\} $. The examples for $n\in\left\{ 2,3,4,5,6\right\} $ also yield finite local complexity (FLC).},
 pdfkeywords={cyclotomic aperiodic substitution tiling; inflation tiling; discrete metric geometry; dense tiling orientations; DTO; CAST; quasiperiodic; nonperiodic; minimal inflation multiplier; minimal scaling factor; substitution matrix; irrational argument; rational argument; local finite complexity; FLC}}

\makeatletter

\pdfpageheight\paperheight
\pdfpagewidth\paperwidth

\providecommand{\tabularnewline}{\\}
\newcommand{\lyxdot}{.}

\numberwithin{equation}{section}
\numberwithin{figure}{section}
  \theoremstyle{definition}
  \newtheorem{defn}{\protect\definitionname}[section]
  \theoremstyle{plain}
  \newtheorem{thm}{\protect\theoremname}[section]
  \theoremstyle{plain}
  \newtheorem{lem}{\protect\lemmaname}[section]
  \theoremstyle{plain}
  \newtheorem{assumption}{\protect\assumptionname}
  \theoremstyle{plain}
  \newtheorem{prop}{\protect\propositionname}[section]
  \theoremstyle{plain}
  \newtheorem{conjecture}{\protect\conjecturename}[section]
  \theoremstyle{remark}
  \newtheorem{rem}{\protect\remarkname}[section]

\usepackage{latexsym}
\usepackage{a4wide}
\pdfoutput=1

\RequirePackage{fix-cm}

\usepackage[T1]{fontenc}
\usepackage{array}
\usepackage{url}
\usepackage{multirow}
\usepackage{amsmath}
\usepackage{booktabs}
\usepackage{graphicx}

\renewcommand{\rho}{\varrho}
\renewcommand{\phi}{\varphi}

\theoremstyle{remark}

\makeatother

  \providecommand{\assumptionname}{Assumption}
  \providecommand{\conjecturename}{Conjecture}
  \providecommand{\definitionname}{Definition}
  \providecommand{\lemmaname}{Lemma}
  \providecommand{\propositionname}{Proposition}
  \providecommand{\remarkname}{Remark}
\providecommand{\theoremname}{Theorem}

\begin{document}

\title{Cyclotomic Aperiodic Substitution Tilings with Dense Tile Orientations}

\author{Stefan Pautze}

\address{Stefan Pautze (Visualien der Breitbandkatze), Am Mitterweg 1, 85309
Pörnbach, Germany}

\urladdr{\texttt{http://www.pautze.de}}

\email{\texttt{stefan@pautze.de}}
\begin{abstract}
The class of Cyclotomic Aperiodic Substitution Tilings (CAST) whose
vertices are supported by the $2n$-th cyclotomic field $\mathbb{Q}\left(\zeta_{2n}\right)$
is extended to cases with Dense Tile Orientations (DTO). It is shown
that every CAST with DTO has an inflation multiplier {\normalsize{}$\eta$}
with irrational argument so that $\nicefrac{k\pi}{2n}\neq\arg\left(\eta\right)\notin\mathbb{\pi Q}$.
The minimal inflation multiplier {\normalsize{}$\eta_{min.irr.}$}
is discussed for $n\geqq2$. Examples of CASTs with DTO, minimal inflation
multiplier{\normalsize{} $\eta_{min.irr.}$} and individual dihedral
symmetry $D_{2n}$ are introduced for $n\in\left\{ 2,3,4,5,6,7\right\} $.
The examples for $n\in\left\{ 2,3,4,5,6\right\} $ also yield finite
local complexity (FLC).
\end{abstract}

\maketitle

\section{Introduction}

Aperiodic substitution tilings and their properties are investigated
and discussed in view of their application in physics and chemistry.
However, they are also of mathematical interest on their own. Without
any doubt, they have great aesthetic qualities as well. See \citep{Grunbaum:1986:TP:19304,oro38933,baake2017aperiodic,HFonl}
and references therein. An interesting variant are substitution tilings
with dense tile orientations (DTO), i.e. the orientations of the tiles
are dense on the circle. Well known examples are the Pinwheel tiling
as described by Radin and Conway in \citep{10.2307/2118575,radinmiles}
and its relatives in \citep{sadun98some,2008PMag...88.2033F,FRETTLOH20081881}.
Recently Frettlöh, Say-awen and de las Peñas described “Substitution
Tilings with Dense Tile Orientations and n-fold Rotational Symmetry”
\citep{FRETTLOH2017120} and continued the discussion in \citep{Say-awen:eo5083}.
For this reason we extend the results in \citep[Chapter 2]{sym9020019}
regarding Cyclotomic Aperiodic Substitution Tilings (CAST) to cases
with DTO in Section~\ref{sec:A_Revised_Definition_Of_CASTs_with_and_without_DTO}.
In Section~\ref{sec:Irrational_Angles} it is shown that every CAST
with DTO has an inflation multiplier $\eta$ with irrational argument
so that $\nicefrac{k\pi}{2n}\neq\arg\left(\eta\right)\notin\mathbb{\pi Q}$.
Furthermore minimal inflation multipliers $\eta_{min.irr.}$ of CASTs
with DTO are derived in Section~\ref{sec:Minimal_Inflation_Multiplier}.
In Section~\ref{sec:Examples_of_CASTs_with_DTO} examples with minimal
inflation multiplier and individual dihedral symmetry $D_{2n}$ are
introduced for the cases $n\in\left\{ 2,3,4,5,6,7\right\} $ and compared
to the results in \citep{FRETTLOH2017120} and \citep{Say-awen:eo5083}.

For terms and definitions we stay close to \citep{oro38933}, \citep{HFonl},
\citep{FRETTLOH2017120}, \citep{Say-awen:eo5083}, \citep{sym9020019}
and references therein:
\begin{itemize}
\item A ``tile'' in $\mathbb{\mathbb{\mathbb{R}}}^{d}$ is defined as
a nonempty compact subset of $\mathbb{\mathbb{\mathbb{R}}}^{d}$ which
is the closure of its interior.
\item A ``tiling'' $\mathcal{T}$ in $\mathbb{\mathbb{\mathbb{R}}}^{d}$
is a countable set of tiles, which is a covering as well as a packing
of $\mathbb{\mathbb{\mathbb{R}}}^{d}$. The union of all tiles is
$\mathbb{\mathbb{\mathbb{R}}}^{d}$. The intersection of the interior
of two different tiles is empty.
\item A ``patch'' $\mathcal{P\mathcal{\subset}T}$ is a finite subset
of a tiling.
\item A tiling $\mathcal{T}$ is called ``aperiodic'' if no translation
maps the tiling to itself.
\item ``Prototiles'' $P_{x}$ serve as building blocks for a tiling.
\item Within this article the term ``substitution'' means, that a tile
is expanded and rotated with a linear map - the complex ``inflation
multiplier'' $\eta\subset\mathbb{C}$ - and dissected into copies
of prototiles in original size - the ``substitution rule''.
\item A ``supertile'' is the result of one or more substitutions, applied
to a single tile.
\item The result of $k$ substitutions applied to a single tile is noted
as ``level-k supertile''. 
\item A substitution is primitive if there is a $k\in\mathbb{N}$ for every
prototile $P_{x}$ such that the corresponding level-$k$ supertile
contains equivalent copies of all prototiles.
\item A substitution tiling $\mathcal{T}$ with $l\in\mathbb{\mathbb{N}}_{>0}$
prototiles and their substitution rules is the result of an infinite
number of substitutions applied to a single prototile.
\item A primitive substitution tiling $\mathcal{T}$ (or more accurately,
the set of substitution rules of its prototiles) is partially characterized
by its primitive substitution matrix $M\in\mathbb{\mathbb{N}}_{0}^{l\times l}$
with a Perron-Frobenius eigenvalue $\lambda_{1}=\eta\overline{\eta}$.
\item A tiling has ``dense tile orientations'' (DTO) if the orientations
of the tiles are dense on the circle.
\item A tiling has DTO if some level-$k$ supertile contains two congruent
tiles which are rotated against each other by an irrational angle
$\varTheta\notin\mathbb{\pi Q}$. 
\item A tiling with ``individual cyclical or dihedral symmetry'' $C_{n}$
or $D_{n}$ contains an infinite number of patches of any size with
cyclical or dihedral symmetry $C_{n}$ or $D_{n}$.
\item A tiling $\mathcal{T}$ with DTO has ``finite local complexity''
(FLC) with respect to rigid motions if for any $r>0$ there are only
finitely many pairwise non-congruent patches in $\mathcal{T}$ fitting
into a ball of radius $r$.
\item We use $\zeta_{n}^{k}\in\mathbb{C}$ to denote the $n$-th roots of
unity so that $\zeta_{n}^{k}=e^{\frac{2\mathtt{i}k\pi}{n}}$ and its
complex conjugate $\overline{\zeta_{n}^{k}}=e^{-\frac{2\mathtt{i}k\pi}{n}}$.
\item $\mathbb{Q}\left(\zeta_{n}\right)$ denotes the $n$-th cyclotomic
field. Since $\mathbb{Q}\left(\zeta_{n}\right)=\mathbb{Q}\left(\zeta_{2n}\right)$
for $odd\;n$ we discuss tilings in the $2n$-th cyclotomic fields
$\mathbb{Q}\left(\zeta_{2n}\right)$ only.
\item The maximal real subfield of $\mathbb{Q}\left(\zeta_{n}\right)$ is
$\mathbb{Q}\left(\zeta_{n}+\overline{\zeta_{n}}\right)$.
\item $\mathbb{Z}\left[\zeta_{n}\right]$ denotes the ring of algebraic
integers in $\mathbb{Q}\left(\zeta_{n}\right)$.
\item $\mathbb{Z}\left[\zeta_{n}+\overline{\zeta_{n}}\right]$ denotes the
the ring of algebraic integers (which are real numbers) in $\mathbb{Q}\left(\zeta_{n}+\overline{\zeta_{n}}\right)$. 
\item We use $\mu_{n,k}$ to denote the $k$-th diagonal of a regular $n$-gon
with side length $1$, in detail $\mu_{n,k}=\sum_{i=0}^{k-1}\zeta_{2n}^{2i-k+1}=\nicefrac{sin(\nicefrac{k\pi}{n})}{sin(\nicefrac{\pi}{n})}$
with $\left\lfloor n/2\right\rfloor \geq k\geq1$ and $\mu_{n,n-k}=\mu_{n,k}$
with $n>k\geq1$.
\item $\mathbb{Z}\left[\mu_{n}\right]=\mathbb{Z}\left[\mu_{n,1},\mu_{n,2},\mu_{n,3}...\mu_{n,\left\lfloor n/2\right\rfloor }\right]$
denotes the ring of the diagonals of a regular $n$-gon.
\item The minimal inflation multiplier $\eta_{min.irr.}$ is the inflation
multiplier $\eta\in\mathbb{Z}\left[\zeta_{2n}\right]\setminus\left\{ 0\right\} $
of a CAST with irrational argument $\arg\left(\eta\right)\notin\pi\mathbb{Q}$
and with minimal absolute value.\pagebreak{}
\end{itemize}

\section{\label{sec:A_Revised_Definition_Of_CASTs_with_and_without_DTO}A
Revised Definition Of CASTs}

We recall the definition of Cyclotomic Aperiodic Substitution Tilings
as noted in \citep{sym9020019}:
\begin{defn}
\label{def:CAST}A (substitution) tiling $\mathcal{T}$ in the complex
plane is cyclotomic if the coordinates of all vertices are algebraic
integers in $\mathbb{Z}\left[\zeta_{2n}\right]$, i.e. an integer
sum of the $2n$-th roots of unity. As a result all vertices of all
substituted prototiles and the inflation multiplier are algebraic
integers as well. 

Recently substitution tilings with dense tile orientations (DTO) and
$n$-fold rotational symmetry have been described in \citep{FRETTLOH2017120}
by Frettlöh, Say-awen and de las Peñas. The results in the paper imply
that a substitution tiling $\mathcal{T}$ with inflation multiplier
$\eta\in\mathbb{Z}\left[\zeta_{2n}\right]\mid\left|\eta\right|>1$
and $\arg\left(\eta\right)\notin\mathbb{\pi Q}$ has vertices supported
by the cyclotomic field $\mathbb{Q}\left(\zeta_{2n}\right)$ but not
necessarily by the ring of algebraic integers $\mathbb{Z}\left[\zeta_{2n}\right]$.
An example is shown in Figure~\ref{fig:Pinwheel}. Obviously Definition~\ref{def:CAST}
is too constrictive.
\end{defn}
\begin{figure}
\begin{center}
\resizebox{1.0\textwidth}{!}{%

\includegraphics{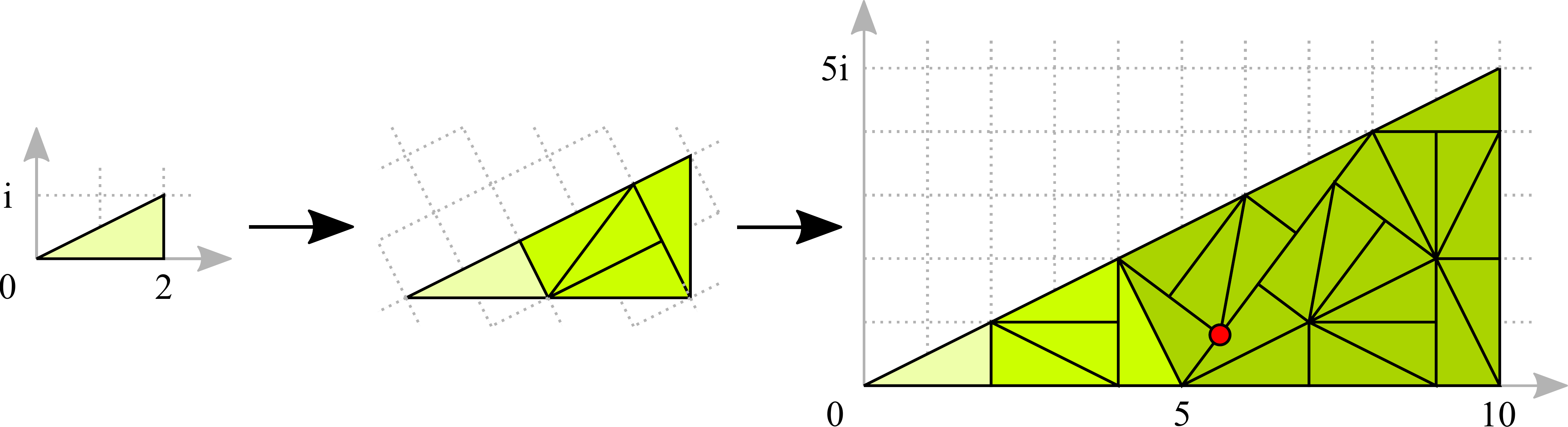}}
\end{center}\caption{\label{fig:Pinwheel}The Pinwheel tiling as described by Radin and
Conway in \citep{10.2307/2118575,radinmiles}, in detail two substitutions
applied to a prototile placed in the complex plane. Despite the initial
position of the prototile not all coordinates of the vertices within
the level-$2$ supertile are elements of $\mathbb{Z}\left[\zeta_{4}\right]$,
e.g. the red marked vertex has coordinates $\frac{28}{5}+i\frac{4}{5}$.}
\end{figure}

\begin{thm}
\label{thm:CAST2}A substitution tiling $\mathcal{T}$ with vertices
supported by the $2n$-th cyclotomic field $\mathbb{Q}\left(\zeta_{2n}\right)$,
\textup{$l\in\mathbb{\mathbb{N}}_{>0}$} prototiles and a primitive
positive substitution matrix $M\in\mathbb{\mathbb{N}}_{0}^{l\times l}$
with a Perron–Frobenius eigenvalue $\lambda_{1}=|\eta^{2}|=\eta\cdot\overline{\eta}$
has an inflation multiplier $\eta\mid\left|\eta\right|>1$ which can
be written as $\frac{z}{\overline{z}}\eta\in\mathbb{Z}\left[\zeta_{2n}\right]$
for some $z\in\mathbb{Q}\left(\zeta_{2n}\right)\setminus\left\{ 0\right\} $
.\end{thm}
\begin{proof}
It is known that the ring of integers $\mathcal{O}_{K}$ of an algebraic
number field $K$ is the ring of all integral elements contained in
$K$. Furthermore an integral element is a root of a monic polynomial
with rational integer coefficients, such as the characteristic polynom
of $M\in\mathbb{\mathbb{N}}_{0}^{l\times l}$. In other words, the
Perron–Frobenius eigenvalue $\lambda_{1}$ of $M$ is an integral
element of $K$ and so $\lambda_{1}\in\mathcal{O}_{K}\subset K$.
Since the tiling $\mathcal{T}$ has vertices supported by the $2n$-th
cyclotomic field $\mathbb{Q}\left(\zeta_{2n}\right)$ we can deduce
that $\eta\in\mathbb{Q}\left(\zeta_{2n}\right)$ and $\lambda_{1}=\eta\cdot\overline{\eta}\in\mathbb{Q}\left(\zeta_{2n}+\overline{\zeta_{2n}}\right)$.
We can note that $K\simeq\mathbb{Q}\left(\zeta_{2n}\right)$ and $\mathcal{O}_{K}\simeq\mathbb{Z}\left[\zeta_{2n}\right]$,
hence $\lambda_{1}=\eta\cdot\overline{\eta}\in\mathbb{Z}\left[\zeta_{2n}+\overline{\zeta_{2n}}\right]$. 

To complete the proof we have to investigate, for which inflation
multipliers $\eta$ the Perron–Frobenius eigenvalue fulfills the conditions
$\lambda_{1}=\eta\cdot\overline{\eta}\in\mathbb{Z}\left[\zeta_{2n}+\overline{\zeta_{2n}}\right]$:

\begin{equation}
\exists z\mid\frac{z}{\overline{z}}\eta\in\mathbb{Z}\left[\zeta_{2n}\right]\rightarrow\lambda_{1}=\eta\cdot\overline{\eta}\in\mathbb{Z}\left[\zeta_{2n}+\overline{\zeta_{2n}}\right]
\end{equation}

\begin{equation}
\nexists z\mid\frac{z}{\overline{z}}\eta\in\mathbb{Z}\left[\zeta_{2n}\right]\rightarrow\lambda_{1}=\eta\cdot\overline{\eta}\notin\mathbb{Z}\left[\zeta_{2n}+\overline{\zeta_{2n}}\right]
\end{equation}

As a result we can conclude:
\begin{equation}
\exists z\mid\frac{z}{\overline{z}}\eta\in\mathbb{Z}\left[\zeta_{2n}\right]\leftrightarrow\lambda_{1}=\eta\cdot\overline{\eta}\in\mathbb{Z}\left[\zeta_{2n}+\overline{\zeta_{2n}}\right]
\end{equation}

\end{proof}
The results in \citep[Chapter 2]{sym9020019} apply also for tilings
in Theorem~\ref{thm:CAST2} and we can revise Definition~\ref{def:CAST}
as follows:
\begin{defn}
\label{def:CAST2}A substitution tiling $\mathcal{T}$ in the complex
plane is cyclotomic if the coordinates of all vertices are supported
by the $2n$-th cyclotomic field $\mathbb{\mathbb{Q}}\left(\zeta_{2n}\right)$.
\end{defn}
For simplification we will limit the following discussion to the case
that $\eta\in\mathbb{Z}\left[\zeta_{2n}\right]$.

\section{\label{sec:Irrational_Angles}Irrational Angles }

In the following chapter we well discuss under which conditions a
CAST may yield DTO.

We recall, all vertices of $\mathcal{T}$ are supported by the $2n$-th
cyclotomic field $\mathbb{\mathbb{Q}}\left(\zeta_{2n}\right)$. So
each finite patch $\mathcal{P}$ which contains at least one copy
of each level-$1$ supertile of each of the $l$ prototiles $P_{x}$
of $\mathcal{T}$ can be mapped to $\mathcal{P}'$ with an integer
multiplier so that the vertices of $\mathcal{P}'$ are supported by
$\mathbb{Z}\left[\zeta_{2n}\right]$. As a consequence each CAST can
be described by a set of prototiles and substitution rules, or more
accurately, level-$1$ supertiles with vertices supported by $\mathbb{Z}\left[\zeta_{2n}\right]$.
A CAST with DTO contains congruent tiles and so also corresponding
edges which are rotated against each other by an irrational angle
$\varTheta\notin\mathbb{\pi Q}$. As consequence at least one prototile
exists whose edges enclose irrational angles. This and the primitivity
of the tiling imply that the inflation multiplier of a CAST with DTO
must have an irrational argument so that $\arg\left(\eta\right)\notin\mathbb{\pi Q}$,
but not necessarily vice versa. An example for the latter case is
shown in Figure~\ref{fig:CAST-without -DTO-and-irr-arg}. 

\begin{figure}[t]
\begin{center}
\resizebox{0.9\textwidth}{!}{%

\includegraphics{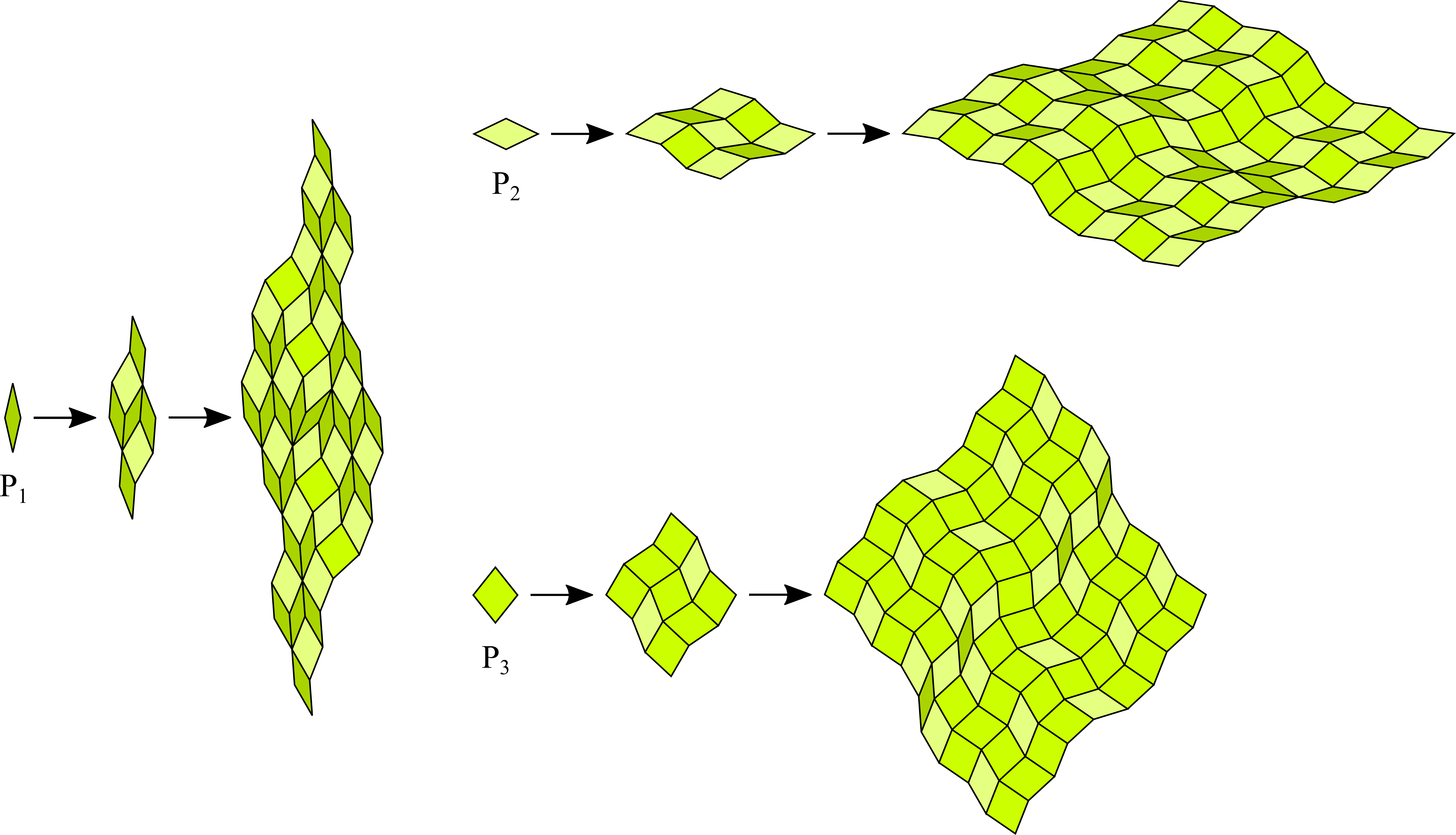}}
\end{center}

\caption{\label{fig:CAST-without -DTO-and-irr-arg}A rhombic CAST with 7, hence
finite many tile orientations. Its inflation multiplier $\eta=\zeta_{14}^{0}+\zeta_{14}^{0}+\zeta_{14}^{1}$
has an irrational argument $\arg\left(\eta\right)\protect\notin\mathbb{\pi Q}$
but its tiles are rotated against each other by rational angles $\nicefrac{k\pi}{7}\in\mathbb{\pi Q}$.}
\end{figure}

While in \citep[Theorem 5]{FRETTLOH2017120} and \citep[Theorem 5]{Say-awen:eo5083}
a parallelogram criterion was used to identify irrational angles,
we need to extend their results to a more general approach. In detail
we need a criterion to decide which elements of $\mathbb{Z}\left[\zeta_{2n}\right]$
have rational or irrational arguments. The author believes (as the
authors of the cited articles) that the following results must be
known already, but he is also not aware of any reference.
\begin{lem}
\label{lem:z_arg(z)_rational_arg(z)_mul_kpi/n}Any $z\in\mathbb{Q}\left(\zeta_{2n}\right)\mid\left|z\right|=1$
on the unit circle which is not a root of unity in $\mathbb{Q}\left(\zeta_{2n}\right)$
has an irrational argument, i.e. $z\in\mathbb{Q}\left(\zeta_{2n}\right)\setminus\left\{ \zeta_{2n}^{k}\right\} \leftrightarrow\arg\left(z\right)\notin\pi\mathbb{Q}$
for \textup{$\left|z\right|=1$}.\end{lem}
\begin{proof}
It is known that any $z\in\mathbb{C}\mid\left|z\right|=1,\;arg\left(z\right)\in\pi\mathbb{Q}$
is a root of unity of some $2n$-th cyclotomic field $\mathbb{Q}\left(\zeta_{2n}\right)$
so that $z=e^{\frac{2\mathtt{i}k\pi}{2n}}=\zeta_{2n}^{k}$ . The $2n$-th
cyclotomic field $\mathbb{Q}\left(\zeta_{2n}\right)$ has by definition
exactly $2n$ roots of unity $\zeta_{2n}^{k}\mid0\leq k<2n$. As a
consequence every other $z\in\mathbb{Q}\left(\zeta_{2n}\right)\setminus\left\{ \zeta_{2n}^{k}\right\} \mid\left|z\right|=1$
on the unit circle must have an irrational argument.\end{proof}
\begin{lem}
\label{lem:z_with_rational_argument}For any $z\mathbb{\in Q}\left(\zeta_{2n}\right)$
with rational argument $\arg\left(z\right)\in\pi\mathbb{Q}$ the argument
of $z$ is an integer multiple of $\frac{\pi}{2n}$, so that $\arg\left(z\right)=\frac{k\pi}{2n}$.
In other words, the elements of $\mathbb{Q}\left(\zeta_{2n}\right)$
with rational argument can be found along $2n$ straight lines in
the complex plane running though $0$ and tilted by \textup{$\frac{k\pi}{2n}$}
relative to the real axis.\end{lem}
\begin{proof}
It is known that $\arg\left(z\right)=\frac{1}{2}\arg\left(\frac{z}{\bar{z}}\right)$.
Because of Lemma~\ref{lem:z_arg(z)_rational_arg(z)_mul_kpi/n} we
know that if the argument of $z$ is rational or not, the same is
true for $\frac{z}{\bar{z}}$ and vice versa, i.e. $\arg\left(\frac{z}{\bar{z}}\right)\notin\pi\mathbb{Q}\leftrightarrow\arg\left(z\right)\notin\pi\mathbb{Q}$
and $\arg\left(\frac{z}{\bar{z}}\right)\in\pi\mathbb{Q}\leftrightarrow\arg\left(z\right)\in\pi\mathbb{Q}$.
The latter term is true if and only if $\frac{z}{\bar{z}}=\zeta_{2n}^{k}$.
Since $\arg\left(\zeta_{2n}^{k}\right)=\frac{k\pi}{n}$ we can conclude
that $\arg\left(z\right)\in\pi\mathbb{Q\leftrightarrow}\arg\left(z\right)=\frac{k\pi}{2n}$. 
\end{proof}
Obviously $\eta$ can be written as sum of $l$ roots of unity in
different ways:
\begin{equation}
\eta=\sum_{k=1}^{l}\zeta_{2n}^{\alpha_{k}}\in\mathbb{Z}\left[\zeta_{2n}\right]\;\;\;\;\;\;(\alpha_{k}\in\mathbb{Z},\;\alpha_{l}-\alpha_{1}<2n,\;\alpha_{k}\leq\alpha_{k+1})\label{eq:lambda-i-alpha}
\end{equation}
\begin{equation}
\eta=\sum_{k=1-n}^{n}a_{k}\zeta_{2n}^{k}\in\mathbb{Z}\left[\zeta_{2n}\right]\;\;\;\;\;\;(a_{k}\in\mathbb{\mathbb{N}}_{0},\;l=\sum_{k=1-n}^{n}a_{k},\;\max(a_{k})>0)\label{eq:lambda-i}
\end{equation}

\begin{assumption}
\label{assu:l_is_minimal}There are multiple ways to describe $\eta$.
We assume $\alpha_{k}$ and $a_{k}$ to be chosen so that the sum
is irreducible, i.e. $l$ is minimal.\end{assumption}
\begin{prop}
\label{prop:Criterion_Eta_is_rational_or_not}Every inflation multiplier
with rational argument $\eta\in\mathbb{Z}\left[\zeta_{2n}\right]\setminus\left\{ 0\right\} \mid\arg\left(\eta\right)\in\pi\mathbb{Q}$
can be noted as in Equation~\eqref{eq:lambda-i-alpha} with $\alpha_{k+1}-\alpha_{k}=\alpha_{l-k+1}-\alpha_{l-k}$
and vice versa.\end{prop}
\begin{proof}
Lemma~\ref{lem:z_with_rational_argument} implies that the absolute
value of $\eta\in\mathbb{Z}\left[\zeta_{2n}\right]\mid\arg\left(\eta\right)\in\pi\mathbb{Q}$
has the following properties: $\left|\eta\right|\mathbb{\in\mathbb{Z}}\left[\zeta_{2n}+\overline{\zeta_{2n}}\right]$
for $\arg\left(\eta\right)=\frac{k\pi}{2n}$ with $even\text{ }k\in\mathbb{Z}$
and $\left|\eta\right|\mathbb{\in\mathbb{Z}}\left[\zeta_{4n}+\overline{\zeta_{4n}}\right]\setminus\mathbb{\mathbb{Z}}\left[\zeta_{2n}+\overline{\zeta_{2n}}\right]$
for $\arg\left(\eta\right)=\frac{k\pi}{2n}$ with $odd~k\in\mathbb{Z}$.
Multiplication of $\left|\eta\right|$ with $e^{\frac{ik\pi}{2n}}$
and rearrangement of the order of the summands will deliver an appropriate
notation for every $k$. Vice versa for every $\eta$ in appropriate
notation a $k\in\mathbb{Z}$ exists so that $\arg\left(\eta e^{-\frac{ik\pi}{2n}}\right)=0$,
hence $\arg\left(\eta\right)=\nicefrac{k\pi}{2n}\in\pi\mathbb{Q}$.
\end{proof}
Proposition~\ref{prop:Criterion_Eta_is_rational_or_not} and Assumption~\ref{assu:l_is_minimal}
can be combined to form a criterion to decide if $\eta$ has a rational
or an irrational argument.

\section{\label{sec:Minimal_Inflation_Multiplier}Minimal Inflation Multipliers
with Irrational Argument}

In the following chapter we will discuss the minimal inflation multiplier
with irrational argument so that $\eta\in\mathbb{Z}\left[\zeta_{2n}\right]\mid\arg\left(\eta\right)\notin\mathbb{\pi Q}$.

We recall the results from \citep[Chapter 2]{sym9020019}: The Perron–Frobenius
eigenvalue $\lambda_{1}\in\mathbb{N}_{0}\left[\mu_{n}\right]$ of
the substitution matrix $M$ of $\mathcal{T}$ can be written as sum
of pairs of complex conjugated roots of unity $\zeta_{2n}^{k}+\overline{\zeta_{2n}^{k}}$
or as sum of diagonals of a regular unit $n$-gon $\mu_{n,k}$ with
additional conditions:
\begin{equation}
\lambda_{1}=|\eta^{2}|=\eta\cdot\overline{\eta}=b_{0}+\sum_{k=1}^{\left\lfloor (n-1)/2\right\rfloor }b_{k}\left(\zeta_{2n}^{k}+\overline{\zeta_{2n}^{k}}\right)\;\;\;\;\;\;(b_{0},b_{k}\in\mathbb{\mathbb{N}}_{0})\label{eq:lambda-a-2}
\end{equation}
\begin{equation}
\lambda_{1}=\sum_{k=1}^{\left\lfloor n/2\right\rfloor }c_{k}\mu_{n,k}\;\;\;\;\;\;(c_{k}\in\mathbb{\mathbb{N}}_{0})\label{eq:lambda-a-3}
\end{equation}
By expanding Equation~\eqref{eq:lambda-a-2} with~\eqref{eq:lambda-i}
we can conclude: 
\begin{equation}
b_{0}=\sum_{k=1-n}^{n}a_{k}^{2}\geq\sum_{k=1-n}^{n}a_{k}=l\label{eq:bo_ge_l}
\end{equation}
Aperiodicity of $\mathcal{T}$ and primitivity of its substitution
matrix for $n\geq4$ are ensured by the following conditions for the
coefficients $c_{k}$:
\begin{equation}
\max\left(c_{k}\right)\geq1\;\;\;\;\;\;(k\geq2;\;odd\;n\geq5)\label{eq:cond_n_odd_c}
\end{equation}
\begin{equation}
\min\left(\left(\max\left(c_{k}\right),\;odd\;k\right),\left(\max\left(c_{k}\right),\;even\;k\right)\right)\geq1\;\;\;\;\;\;(even\;n\geq4)\label{eq:cond_n_even_c}
\end{equation}
The same properties can be ensured by the following conditions for
the coefficients $b_{k}$:
\begin{equation}
b_{k}\geq b_{k+2}\;\;\;\;\;\;(\left\lfloor (n-1)/2\right\rfloor \geq k+2>k\geq0)\label{eq:cond_b}
\end{equation}
\begin{equation}
\max\left(b_{1},b_{2}\right)\geq1\;\;\;\;\;\;(odd\;n\geq5)\label{eq:cond_n_odd_b}
\end{equation}
\begin{equation}
\min\left(b_{0},b_{1}\right)\geq1\;\;\;\;\;\;(even\;n\geq4)\label{eq:cond_n_even_b}
\end{equation}

\begin{lem}
\label{lem:irrat_arg_need_l_ge_3}Since $\arg\left(\zeta_{2n}^{\alpha{}_{1}}+\zeta_{2n}^{\alpha{}_{2}}\right)=\frac{\arg\left(\zeta_{2n}^{\alpha{}_{1}}\right)+\arg\left(\zeta_{2n}^{\alpha{}_{2}}\right)}{2}\in\mathbb{\pi Q}$
with $\alpha{}_{1}\neq-\alpha{}_{2}$ the minimal inflation multiplier
\textup{$\eta_{min.irr.}$} must be a sum of $l\geq3$ roots of unity
to ensure\textup{ $\arg\left(\eta_{min.irr.}\right)\notin\mathbb{\pi Q}$}.
\end{lem}
For $odd\;n\geqq5$ we use Equations~\eqref{eq:lambda-a-2}, \eqref{eq:bo_ge_l},
\eqref{eq:cond_b}, \eqref{eq:cond_n_odd_b} and Lemma~\ref{lem:irrat_arg_need_l_ge_3}
to identify the minimal Perron-Frobenius eigenvalue $\lambda_{1\;min.irr.}=\eta_{min.irr.}\cdot\overline{\eta_{min.irr.}}$.
We can conclude that $b_{0}\geqq3$ and either $b_{1}\geqq1$ or $b_{2}\geqq1$.
As a consequence the first candidates for $\lambda_{1\;min.irr.}$
to be checked are the terms $3+\zeta_{2n}^{1}+\overline{\zeta_{2n}^{1}}$
and $3+\zeta_{2n}^{2}+\overline{\zeta_{2n}^{2}}$.

For $even\;n\geqq4$ we use a similar approach using Equations~\eqref{eq:lambda-a-2},
\eqref{eq:bo_ge_l}, \eqref{eq:cond_b}, \eqref{eq:cond_n_even_b}
and Lemma~\ref{lem:irrat_arg_need_l_ge_3}. Here we can conclude
that $b_{0}\geqq3$ and $b_{1}\geqq1$, so the first candidate for
$\lambda_{1\;min.irr.}$ to be checked is the term $3+\zeta_{2n}^{1}+\overline{\zeta_{2n}^{1}}$.

If no appropriate $\eta$ exists for the initial candidates then in
both cases Equation~\eqref{eq:cond_b} provides a guideline to identify
the next candidates by adding $\zeta_{2n}^{0}=1$ or an appropriate
$\zeta_{2n}^{k}$+$\overline{\zeta_{2n}^{k}}$.

The initial candidates for $\lambda_{1}$ as well as all other candidates
with $b_{0}=3$ can be checked with the following equation for $\lambda_{1}$.
For this purpose we expand Equation~\eqref{eq:lambda-a-2} with~\eqref{eq:lambda-i-alpha}
and $l=3$:
\begin{equation}
\lambda_{1}=\eta\overline{\eta}=3+\zeta_{2n}^{\alpha{}_{2}-\alpha{}_{1}}+\overline{\zeta_{2n}^{\alpha{}_{2}-\alpha{}_{1}}}+\zeta_{2n}^{\alpha{}_{3}-\alpha{}_{1}}+\overline{\zeta_{2n}^{\alpha{}_{3}-\alpha{}_{1}}}+\zeta_{2n}^{\alpha{}_{3}-\alpha{}_{2}}+\overline{\zeta_{2n}^{\alpha{}_{3}-\alpha{}_{2}}}
\end{equation}

The check itself can be completed by solving an equation system which
may or may not have solutions for $\alpha_{1},\;\alpha_{2},\;\alpha_{3}$
which determine the inflation multiplier $\eta=3+\zeta_{2n}^{\alpha{}_{1}}+\zeta_{2n}^{\alpha{}_{2}}+\zeta_{2n}^{\alpha{}_{3}}$.

For $l\geqq4$ we we will use a numerical approach. A computer algorithm
is designed to decide if an $\eta=\sum_{k=1}^{l}\zeta_{2n}^{\alpha_{k}}$
exists for arbitrary $\lambda_{1}\in\mathbb{N}_{0}\left[\mu_{n}\right]$,
$l\geqq3$ and~$n\geqq2$ by checking all combinations of $\alpha_{1},\alpha_{2},\alpha_{3}...\alpha_{l}$
with some reasonable optimisations. 

It is known that for $n\in\left\{ 2,3\right\} $ the sets of algebraic
integers in the $4$th and $6$th cyclotomic field $\mathbb{Z}\left[\zeta_{4}\right]$
and $\mathbb{Z}\left[\zeta_{6}\right]$ form lattices - uniform, discrete
and periodic point sets - in the complex plane $\mathbb{C}$. In detail
$\mathbb{Z}\left[\zeta_{4}\right]$ is equivalent to the set of Euler
integers and $\mathbb{Z}\left[\zeta_{6}\right]$ is equivalent to
the Eisenstein integers. We investigate the lattice points close to
$0$ and find $\eta_{min.irr.}=\zeta_{2n}^{0}+\zeta_{2n}^{0}+\zeta_{2n}^{1}\mid n\in\left\{ 2,3\right\} $.

For $n=4$ we can show numerically that $\eta=\zeta_{8}^{0}+\zeta_{8}^{0}+\zeta_{8}^{1}$
and $\lambda_{1}=6+2\left(\zeta_{8}^{1}+\overline{\zeta_{8}^{1}}\right)$
are minima under preceding conditions for $l=3$. Furthermore we can
show that for $l=4$ and $0\leqq\alpha_{1}<\alpha_{2}<\alpha_{3}<\alpha_{4}<2n$
no $\lambda_{1}$ exist which fulfills the preceding conditions, in
detail $\arg\left(\zeta_{8}^{0}+\zeta_{8}^{1}+\zeta_{8}^{2}+\zeta_{8}^{3}\right)=\nicefrac{3\pi}{8}\in\pi\mathbb{Q}$
and $\left|\zeta_{8}^{0}+\zeta_{8}^{1}+\zeta_{8}^{3}+\zeta_{8}^{6}\right|^{2}=4-\left(\zeta_{8}^{1}+\overline{\zeta_{8}^{1}}\right)=4-\mu_{4,2}\notin\mathbb{N}\left(\mu_{4}\right)$.
As a consequence at least two $\alpha{}_{k}$ must have identical
values which is also true all $l\geqq5$. As a consequence we note
that $\lambda_{1}\geqq6+\zeta_{8}^{1}+\overline{\zeta_{8}^{1}}$ for
$l\geqq4$. We check the lower boundary which is smaller than the
minimum of $\lambda_{1}$ for $l=3$ and find $\eta_{min.irr.}=\zeta_{8}^{0}+\zeta_{8}^{1}+\zeta_{8}^{1}+\zeta_{8}^{3}$
with $\lambda_{1\;min.irr.}=6+\zeta_{8}^{1}+\overline{\zeta_{8}^{1}}$.

For $n=6$ we have $\zeta_{12}^{2}+\overline{\zeta_{12}^{2}}=1$ and
so we can note $3+\zeta_{12}^{1}+\overline{\zeta_{12}^{1}}<3+\zeta_{12}^{1}+\overline{\zeta_{12}^{1}}+\zeta_{12}^{2}+\overline{\zeta_{12}^{2}}=4+\zeta_{12}^{1}+\overline{\zeta_{12}^{1}}<3+\zeta_{12}^{1}+\overline{\zeta_{12}^{1}}+2\left(\zeta_{12}^{2}+\overline{\zeta_{12}^{2}}\right)=5+\zeta_{12}^{1}+\overline{\zeta_{12}^{1}}$.
For the first candidate we already know that no solution for $\eta$
exists. For the second and third candidate we find $\eta_{min.irr.}=\zeta_{12}^{0}+\zeta_{12}^{1}+\zeta_{12}^{3}$.

For $n=8$ we have $\zeta_{16}^{3}+\overline{\zeta_{16}^{3}}<1<\zeta_{16}^{2}+\overline{\zeta_{16}^{2}}$
and so we can note $3+\zeta_{16}^{1}+\overline{\zeta_{16}^{1}}<3+\zeta_{16}^{1}+\overline{\zeta_{16}^{1}}+\zeta_{16}^{3}+\overline{\zeta_{16}^{3}}<4+\zeta_{16}^{1}+\overline{\zeta_{16}^{1}}<3+\zeta_{16}^{1}+\overline{\zeta_{16}^{1}}+\zeta_{16}^{2}+\overline{\zeta_{16}^{2}}+\zeta_{16}^{3}+\overline{\zeta_{16}^{3}}$.
For the first candidate we already know that no solution for $\eta$
exists. For the second candidate we find $\eta_{min.irr.}=\zeta_{16}^{0}+\zeta_{16}^{1}+\zeta_{16}^{4}$.

For $n=12$ we have $\zeta_{24}^{4}+\overline{\zeta_{24}^{4}}=1$
and so we can note $3+\zeta_{24}^{1}+\overline{\zeta_{24}^{1}}<4+\zeta_{24}^{1}+\overline{\zeta_{24}^{1}}<3+\zeta_{24}^{1}+\overline{\zeta_{24}^{1}}+\zeta_{24}^{3}+\overline{\zeta_{24}^{3}}<4+\zeta_{24}^{1}+\overline{\zeta_{24}^{1}}+\zeta_{24}^{3}+\overline{\zeta_{24}^{3}}=3+\zeta_{24}^{1}+\overline{\zeta_{24}^{1}}+\zeta_{24}^{3}+\overline{\zeta_{24}^{3}}+\zeta_{24}^{4}+\overline{\zeta_{24}^{4}}<5+\zeta_{24}^{1}+\overline{\zeta_{24}^{1}}$.
For the first candidate we already know that no solution for $\eta$
exists. For the second candidate we can show that numerically. For
the third term no solution exists. For the fourth and fifth candidates
we find $\eta_{min.irr.}=\zeta_{24}^{0}+\zeta_{24}^{1}+\zeta_{24}^{4}$.

For $even\;n\geqq10$ with $n\mathrm{\;mod\;}4=2$ we have $\zeta_{2n}^{3}+\overline{\zeta_{2n}^{3}}>1$
and so we can note $3+\zeta_{2n}^{1}+\overline{\zeta_{2n}^{1}}<4+\zeta_{2n}^{1}+\overline{\zeta_{2n}^{1}}<3+\zeta_{2n}^{1}+\overline{\zeta_{2n}^{1}}+\zeta_{2n}^{3}+\overline{\zeta_{2n}^{3}}$.
For the first candidate we already know that no solution for $\eta$
exists. For the second candidate we find $\eta_{min.irr.}=\zeta_{2n}^{0}+\zeta_{2n}^{1}+\zeta_{2n}^{(n+2)/4}+\zeta_{2n}^{(3n+2)/4}$.

For $even\;n\geqq16$ with $n\mathrm{\;mod\;}4=0$ we have $\zeta_{2n}^{3}+\overline{\zeta_{2n}^{3}}>1$
and so we can note $3+\zeta_{2n}^{1}+\overline{\zeta_{2n}^{1}}<4+\zeta_{2n}^{1}+\overline{\zeta_{2n}^{1}}<3+\zeta_{2n}^{1}+\overline{\zeta_{2n}^{1}}+\zeta_{2n}^{3}+\overline{\zeta_{2n}^{3}}<5+\zeta_{2n}^{1}+\overline{\zeta_{2n}^{1}}<4+\zeta_{2n}^{1}+\overline{\zeta_{2n}^{1}}+\zeta_{2n}^{3}+\overline{\zeta_{2n}^{3}}<6+\zeta_{2n}^{1}+\overline{\zeta_{2n}^{1}}<3+\zeta_{2n}^{1}+\overline{\zeta_{2n}^{1}}+\zeta_{2n}^{2}+\overline{\zeta_{2n}^{2}}+\zeta_{2n}^{3}+\overline{\zeta_{2n}^{3}}$.
For the first and the third candidate no solution for $\eta$ exists.
For the second and the forth candidate no solution for $\eta$ were
found.
\begin{conjecture}
\label{conj:no_eta_for_4plusmu2_or_5plusmu2}For \textup{$even\;n\geqq16$
with $n\;\mathrm{mod\;}4=0$} no \textup{$\eta\in\mathbb{Z}\left[\zeta_{2n}\right]$}
exists so that \textup{$\lambda_{1}=\eta\overline{\eta}=4+\zeta_{2n}^{1}+\overline{\zeta_{2n}^{1}}$}
or \textup{$\lambda_{1}=\eta\overline{\eta}=5+\zeta_{2n}^{1}+\overline{\zeta_{2n}^{1}}$}.
\end{conjecture}
The status of Conjecture~\ref{conj:no_eta_for_4plusmu2_or_5plusmu2}
is subject to further research. However, for $n\leqq100$ it was confirmed
numerically. For the fifth term we find $\eta_{min.irr.}=\zeta_{2n}^{0}+\zeta_{2n}^{1}+\zeta_{2n}^{(n+2)/4}+\zeta_{2n}^{(3n+2)/4}$.

For $n=5$ we have $1+\zeta_{10}^{2}+\overline{\zeta_{10}^{2}}=\zeta_{10}^{1}+\overline{\zeta_{10}^{1}}$
and so we can note $2+\zeta_{10}^{1}+\overline{\zeta_{10}^{1}}=3+\zeta_{10}^{2}+\overline{\zeta_{10}^{2}}<3+2\left(\zeta_{10}^{2}+\overline{\zeta_{10}^{2}}\right)<3+\zeta_{10}^{1}+\overline{\zeta_{10}^{1}}=4+\zeta_{10}^{2}+\overline{\zeta_{10}^{2}}<3+3\left(\zeta_{10}^{2}+\overline{\zeta_{10}^{2}}\right)$.
For the first and second candidate term we find $\eta=\zeta_{10}^{0}+\zeta_{10}^{1}$
with rational argument: $\arg\left(\zeta_{10}^{0}+\zeta_{10}^{1}\right)=\nicefrac{\pi}{1\text{0}}$.
For the third term no solution exists. For the fourth candidate we
find $\eta_{min.irr.}=\zeta_{10}^{0}+\zeta_{10}^{1}+\zeta_{10}^{3}$.

For $odd\;n\geqq7$ we have $\zeta_{2n}^{2}+\overline{\zeta_{2n}^{2}}>1$
and so we can note $3+\zeta_{2n}^{2}+\overline{\zeta_{2n}^{2}}<3+\zeta_{2n}^{1}+\overline{\zeta_{2n}^{1}}<4+\zeta_{2n}^{2}+\overline{\zeta_{2n}^{2}}<3+2\left(\zeta_{2n}^{2}+\overline{\zeta_{2n}^{2}}\right)$.
For $odd\;n\geqq9$ we can additionally note $3+\zeta_{2n}^{1}+\overline{\zeta_{2n}^{1}}<3+\zeta_{2n}^{2}+\overline{\zeta_{2n}^{2}}+\zeta_{2n}^{4}+\overline{\zeta_{2n}^{4}}$.
As a consequence there are no smaller candidates than the initial
candidates. For $3+\zeta_{2n}^{2}+\overline{\zeta_{2n}^{2}}$ we find
$\alpha_{2}-\alpha_{1}=2$ and $\alpha_{3}-\alpha_{1}=\nicefrac{n}{2}+1\notin\mathbb{Z}$
for $odd\;n$, so it is not a valid solution. For $3+\zeta_{2n}^{2}+\overline{\zeta_{2n}^{2}}$
we find $\alpha_{2}-\alpha_{1}=1$ and $\alpha_{3}-\alpha_{1}=\nicefrac{(n+1)}{2}\in\mathbb{Z}$
for $odd\;n$.

As a consequence the solution $\eta_{min.irr.}=\zeta_{2n}^{0}+\zeta_{2n}^{1}+\zeta_{2n}^{\nicefrac{(n+1)}{2}}$
applies for $odd\;n\geqq5$ in general. 

A summary of the results for $\eta_{min.irr.}$ and $\lambda_{1}$
can be found in Table~\ref{tab:min-infl-with-dto}.

All results for $\eta_{min.irr.}$ meet the conditions in Proposition~\ref{prop:Criterion_Eta_is_rational_or_not}
and Assumption~\ref{assu:l_is_minimal}, so that $\arg\left(\eta_{min.irr.}\right)\notin\mathbb{\pi Q}$.
\begin{rem}
Because of the symmetry properties of $\mathbb{Z}\left[\zeta_{2n}\right]$
in detail dihedral symmetry $D_{2n}$ and $\arg\left(\eta_{min.irr.}\right)\notin\pi\mathbb{Q}$
for this and the following cases in total always $4n$ solutions for
$\eta_{min.irr.}$ exist. For simplification we just note one representative
solution.
\end{rem}
\begin{table}
\caption{\label{tab:min-infl-with-dto}Minimal inflation multipliers $\eta_{min.irr.}$
of CASTs with DTO and vertices supported by the $2n$-th cyclotomic
field $\mathbb{Q}\left(\zeta_{2n}\right)$}

\begin{center}
\resizebox{\textwidth}{!}{%

\begin{tabular}{|>{\centering}m{0.005\textwidth}|>{\centering}m{0.15\textwidth}|>{\centering}m{0.33\textwidth}|>{\centering}m{0.4\textwidth}|>{\centering}m{0.17\textwidth}|}
\hline 
\multicolumn{2}{|c|}{$n$} & Minimal inflation multiplier $\eta_{min.irr.}$ & Minimal Perron-Frobenius eigenvalue $\lambda_{1\;min.irr.}=\eta_{min.irr.}\overline{\eta_{min.irr.}}$ & Approach\tabularnewline
\hline 
\hline 
\multicolumn{2}{|c|}{$2$} & $2+\zeta_{4}^{1}$ & $5$ & analytic\tabularnewline
\hline 
\multicolumn{2}{|c|}{$3$} & $2+\zeta_{6}^{1}$ & $3$ & analytic\tabularnewline
\hline 
\multicolumn{2}{|c|}{$4$} & $1+2\zeta_{8}^{1}+\zeta_{8}^{3}$ & $6+\zeta_{8}^{1}+\overline{\zeta_{8}^{1}}=6+\mu_{4,2}$ & numeric\tabularnewline
\hline 
\multicolumn{2}{|c|}{$6$} & $1+\zeta_{12}^{1}+\zeta_{12}^{3}$ & $4+\zeta_{12}^{1}+\overline{\zeta_{12}^{1}}=4+\mu_{6,2}$ & analytic\tabularnewline
\hline 
\multicolumn{2}{|c|}{$8$} & $1+\zeta_{16}^{1}+\zeta_{16}^{4}$ & $3+\zeta_{16}^{1}+\overline{\zeta_{16}^{1}}+\zeta_{16}^{3}+\overline{\zeta_{16}^{3}}=3+\mu_{8,4}$ & analytic\tabularnewline
\hline 
\multicolumn{2}{|c|}{$12$} & $1+\zeta_{24}^{1}+\zeta_{24}^{4}$ & $4+\zeta_{24}^{1}+\overline{\zeta_{24}^{1}}+\zeta_{24}^{3}+\overline{\zeta_{24}^{3}}=4+\mu_{12,4}$ & analytic\tabularnewline
\hline 
\multicolumn{2}{|c|}{$odd\;n\geqq5$} & $1+\zeta_{2n}^{1}+\zeta_{2n}^{(n+1)/2}$ & $3+\zeta_{2n}^{1}+\overline{\zeta_{2n}^{1}}=3+\mu_{n,2}$ & analytic\tabularnewline
\cline{2-5} 
 & $5$ & $1+\zeta_{10}^{1}+\zeta_{10}^{3}$ & $3+\zeta_{10}^{1}+\overline{\zeta_{10}^{1}}=3+\mu_{5,2}$ & analytic\tabularnewline
\cline{2-5} 
 & $7$ & $1+\zeta_{14}^{1}+\zeta_{14}^{4}$ & $3+\zeta_{14}^{1}+\overline{\zeta_{14}^{1}}=3+\mu_{7,2}$ & analytic\tabularnewline
\cline{2-5} 
 & $9$ & $1+\zeta_{18}^{1}+\zeta_{18}^{5}$ & $3+\zeta_{18}^{1}+\overline{\zeta_{18}^{1}}=3+\mu_{9,2}$ & analytic\tabularnewline
\cline{2-5} 
 & \multicolumn{4}{c|}{$\ldots$}\tabularnewline
\hline 
\multicolumn{1}{|>{\centering}m{0.005\textwidth}}{} & $even\;n\geqq10$ \linebreak $n\;\mathrm{mod\;}4=2$ & $1+\zeta_{2n}^{1}+\zeta_{2n}^{(n+2)/4}+\zeta_{2n}^{(3n+2)/4}$ & $4+\zeta_{2n}^{1}+\overline{\zeta_{2n}^{1}}=4+\mu_{n,2}$ & analytic\tabularnewline
\cline{2-5} 
 & $10$ & $1+\zeta_{20}^{1}+\zeta_{20}^{3}+\zeta_{20}^{8}$ & $4+\zeta_{20}^{1}+\overline{\zeta_{20}^{1}}=4+\mu_{10,2}$ & analytic\tabularnewline
\cline{2-5} 
 & $14$ & $1+\zeta_{28}^{1}+\zeta_{28}^{4}+\zeta_{28}^{11}$ & $4+\zeta_{28}^{1}+\overline{\zeta_{28}^{1}}=4+\mu_{14,2}$ & analytic\tabularnewline
\cline{2-5} 
 & $18$ & $1+\zeta_{36}^{1}+\zeta_{36}^{5}+\zeta_{36}^{14}$ & $4+\zeta_{36}^{1}+\overline{\zeta_{36}^{1}}=4+\mu_{18,2}$ & analytic\tabularnewline
\cline{2-5} 
 & \multicolumn{4}{c|}{$\ldots$}\tabularnewline
\hline 
\multicolumn{1}{|>{\centering}m{0.005\textwidth}}{} & $even\;n\geqq16$ \linebreak $n\;\mathrm{mod\;}4=0$ & $1+\zeta_{2n}^{1}+\zeta_{2n}^{(n-2)/2}+\zeta_{2n}^{(n+4)/2}$ & $4+\zeta_{2n}^{1}+\overline{\zeta_{2n}^{1}}+\zeta_{2n}^{3}+\overline{\zeta_{2n}^{3}}=4+\mu_{n,4}$ & numeric for $16\leqq n\leqq100$,

conjecture for $n>100$\tabularnewline
\cline{2-5} 
 & $16$ & $1+\zeta_{32}^{1}+\zeta_{32}^{7}+\zeta_{32}^{10}$ & $4+\zeta_{32}^{1}+\overline{\zeta_{32}^{1}}+\zeta_{32}^{3}+\overline{\zeta_{32}^{3}}=4+\mu_{16,4}$ & numeric\tabularnewline
\cline{2-5} 
 & $20$ & $1+\zeta_{40}^{1}+\zeta_{40}^{9}+\zeta_{40}^{12}$ & $4+\zeta_{40}^{1}+\overline{\zeta_{40}^{1}}+\zeta_{40}^{3}+\overline{\zeta_{40}^{3}}=4+\mu_{20,4}$ & numeric\tabularnewline
\cline{2-5} 
 & $24$ & $1+\zeta_{48}^{1}+\zeta_{48}^{11}+\zeta_{48}^{14}$ & $4+\zeta_{48}^{1}+\overline{\zeta_{48}^{1}}+\zeta_{48}^{3}+\overline{\zeta_{48}^{3}}=4+\mu_{24,4}$ & numeric\tabularnewline
\cline{2-5} 
 & \multicolumn{4}{c|}{$\ldots$}\tabularnewline
\hline 
\end{tabular}

}

\end{center}
\end{table}

\section{\label{sec:Examples_of_CASTs_with_DTO}Examples of CASTs with DTO
and Minimal Inflation Multiplier}

In this chapter we will introduce substitution rules of CASTs with
DTO for $n\in\{2,3,4,5,6,7\}$ and minimal inflation multiplier as
shown in Figures~\ref{fig:CAST_DTO_2}~-~\ref{fig:CAST_DTO_7}
and compare our results with those in \citep{FRETTLOH2017120} and
\citep{Say-awen:eo5083}.

Within this paper $n$ is used to describe the $2n$-th cyclotomic
fields $\mathbb{Q}\left(\zeta_{2n}\right)$ which support the vertices
of the tilings, while \citep{FRETTLOH2017120} and \citep{Say-awen:eo5083}
use it to describe tilings invariant under $n$-fold rotation. To
avoid confusion we use the notation $n^{*}$ for the latter purpose.
With $\mathbb{Q}\left(\zeta_{n}\right)\mathbb{=Q}\left(\zeta_{2n}\right)$
for $odd\;n$ it is obvious that $n^{*}=n$ for $odd\;n^{*}$ and
$n^{*}=2n$ for $even\;n^{*}$. This is especially important for comparing
results, such as symmetry conditions and inflation multipliers:

The tilings in \citep{FRETTLOH2017120} and \citep{Say-awen:eo5083}
yield individual cyclical symmetry $C_{^{n*}}$, while the tilings
and substitution rules in this article yield individual dihedral symmetry
$D_{2n}$.

The inflation multiplier of the tilings in \citep{FRETTLOH2017120}
and \citep{Say-awen:eo5083} can be rewritten as $\eta=\zeta_{2n}^{0}+\zeta_{2n}^{0}+\zeta_{2n}^{1}$.
By comparing the value with the results in Table~\ref{tab:min-infl-with-dto}
we note that $\left|\eta_{min.irr.}\right|<\left|\zeta_{2n}^{0}+\zeta_{2n}^{0}+\zeta_{2n}^{1}\right|$
for $n\geqq4$ and $\left|\eta_{min.irr.}\right|=\left|\zeta_{2n}^{0}+\zeta_{2n}^{0}+\zeta_{2n}^{1}\right|$
for $n\in\left\{ 2,3\right\} $.

Furthermore the tilings in \citep{FRETTLOH2017120} and \citep{Say-awen:eo5083}
were designed to yield finite local complexity (FLC). As noted in
\citep[Remark 2.6]{frettloeh_and_richard_2014}, \citep{FRETTLOH2017120}
and \citep[Lemma 7]{Say-awen:eo5083} a substitution tiling $\mathcal{T}$
with respect to a substitution with a finite prototile set consisting
of polygons satisfying the condition that the tiles in the level-$k$
supertiles meet ``full-edge to full-edge'' for all  $k\in\mathbb{\mathbb{N}}_{>0}$
for all prototiles $P_{x}$ implies that the tiling $\mathcal{T}$
has FLC.

The substitution rules of CASTs with $n\in\{2,3,4,5,6\}$ in Figures~\ref{fig:CAST_DTO_2}~-~\ref{fig:CAST_DTO_6}
satisfy this conditions perfectly. As the substitution rules in \citep{FRETTLOH2017120}
and \citep{Say-awen:eo5083} they meet the ``full-edge to full-edge''
condition and all inflated edges of the same length are dissected
in the same manner. In addition and contrast to \citep{FRETTLOH2017120}
and \citep{Say-awen:eo5083} the dissections of the inflated edges
herein are strictly mirror symmetric, so indication of orientations
can be omitted in general.

\begin{figure}
\begin{center}
\resizebox{0.75\textwidth}{!}{%

\includegraphics{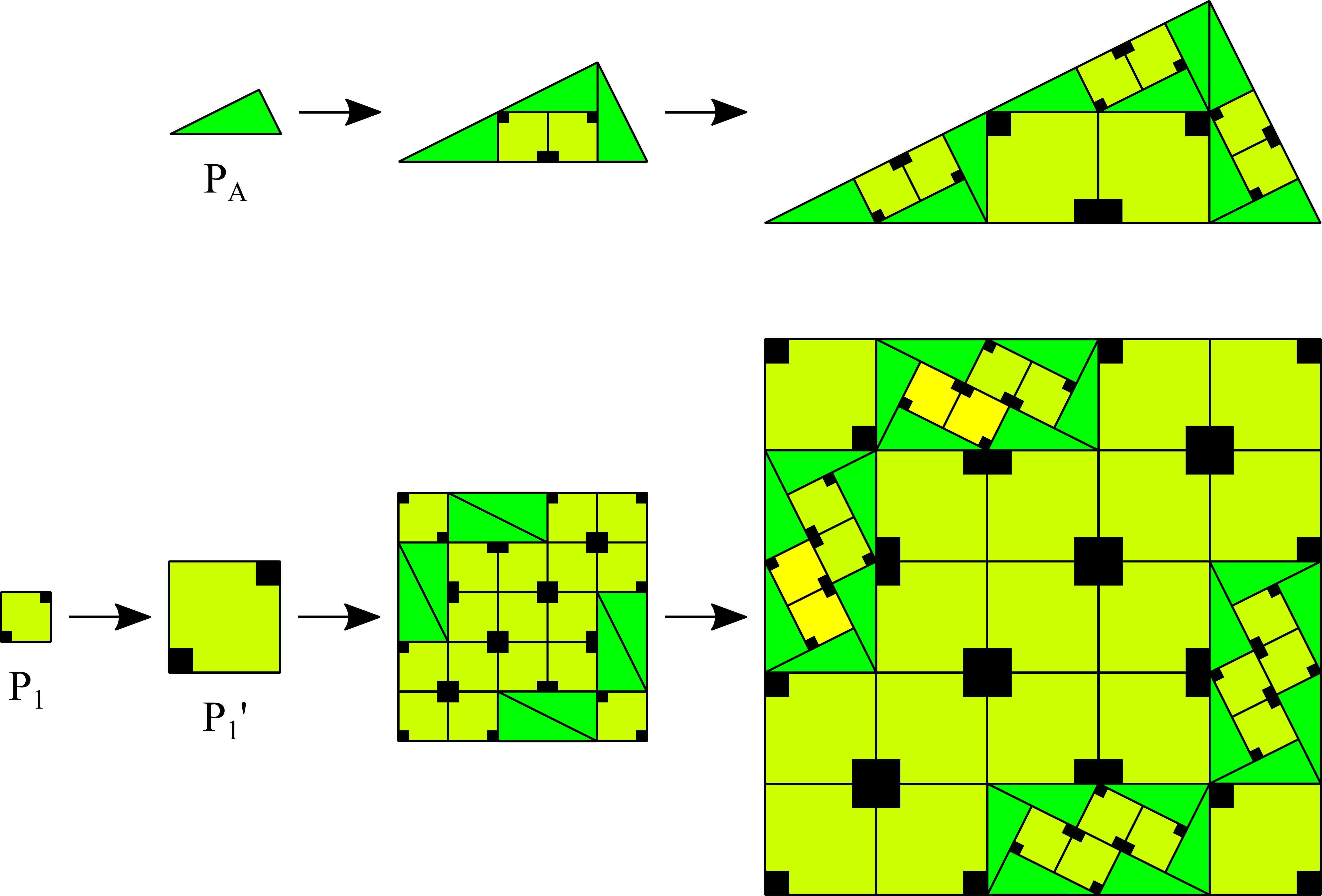}}
\end{center}\caption{\label{fig:CAST_DTO_2}Substitution rules of a CAST with DTO and FLC
for the case $n=2$ with minimal inflation multiplier and individual
dihedral symmetry $D_{4}$. Here the $P_{1}$ level-$3$ supertile
contains two congruent patches (highlighted in yellow) which are rotated
against each other by an irrational angle.}
\end{figure}

\begin{figure}
\begin{center}
\resizebox{0.85\textwidth}{!}{%

\includegraphics{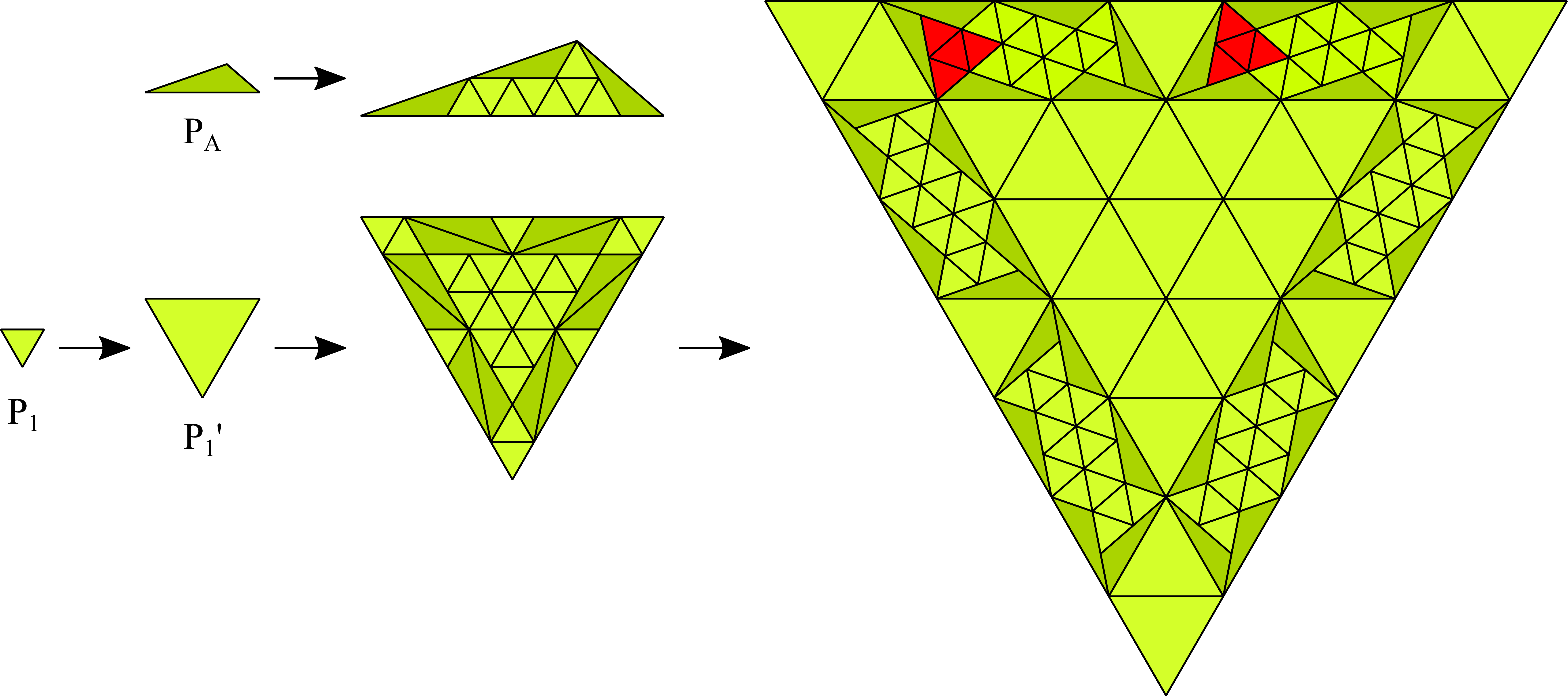}}
\end{center}\caption{\label{fig:CAST_DTO_3}Substitution rules of a CAST with DTO and FLC
for the case $n=3$ with minimal inflation multiplier and individual
dihedral symmetry $D_{6}$. Here the $P_{1}$ level-$3$ supertile
contains two congruent patches (highlighted in red) which are rotated
against each other by an irrational angle.}
\end{figure}

\begin{figure}
\begin{center}
\resizebox{1.0\textwidth}{!}{%

\includegraphics{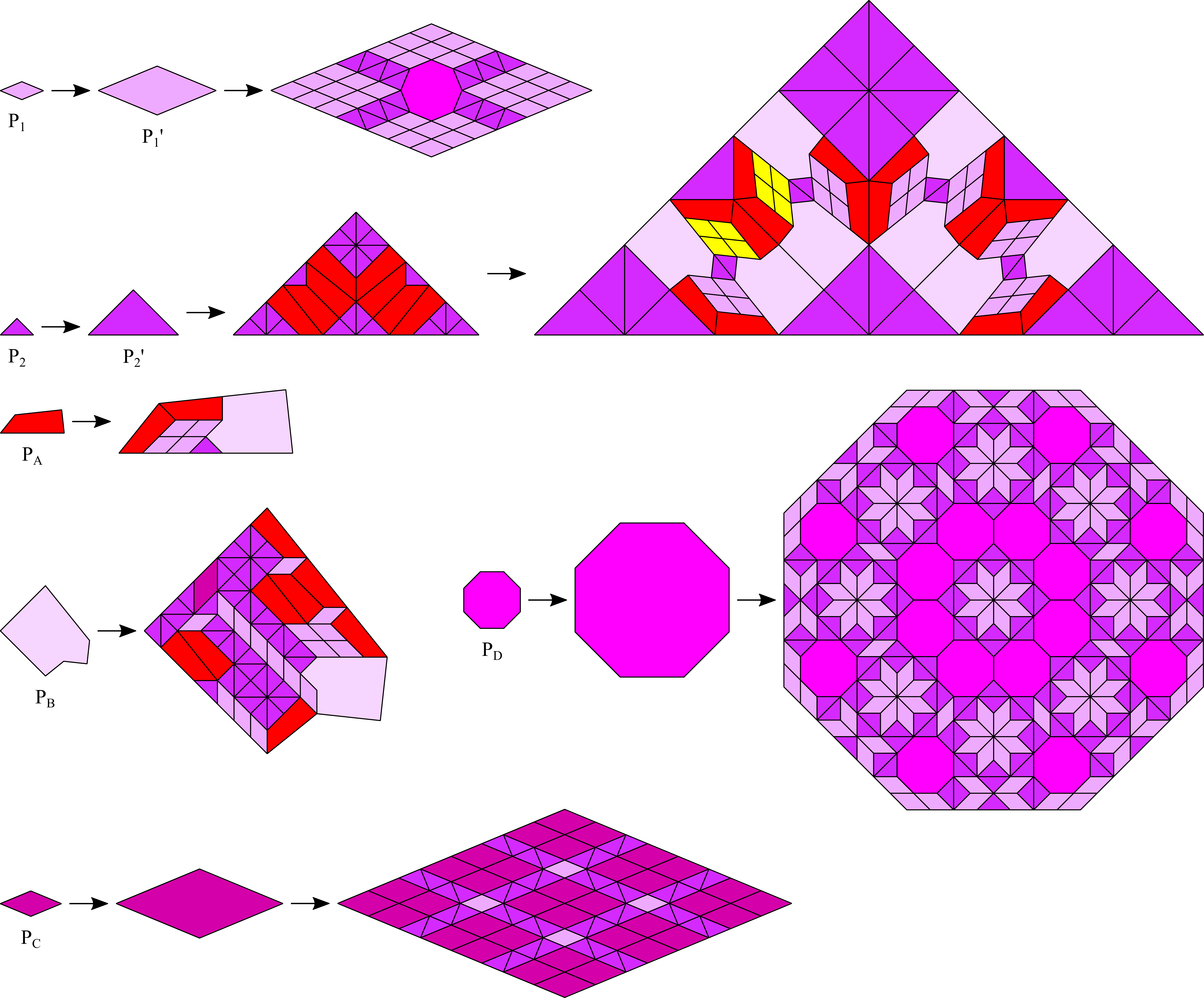}}
\end{center}\caption{\label{fig:CAST_DTO_4}Substitution rules of a CAST with DTO and FLC
for the case $n=4$ with minimal inflation multiplier and individual
dihedral symmetry $D_{8}$. Here the $P_{2}$ level-$3$ supertile
contains two congruent patches (highlighted in yellow) which are rotated
against each other by an irrational angle.}
\end{figure}

\begin{figure}
\begin{center}
\resizebox{0.9\textwidth}{!}{%

\includegraphics{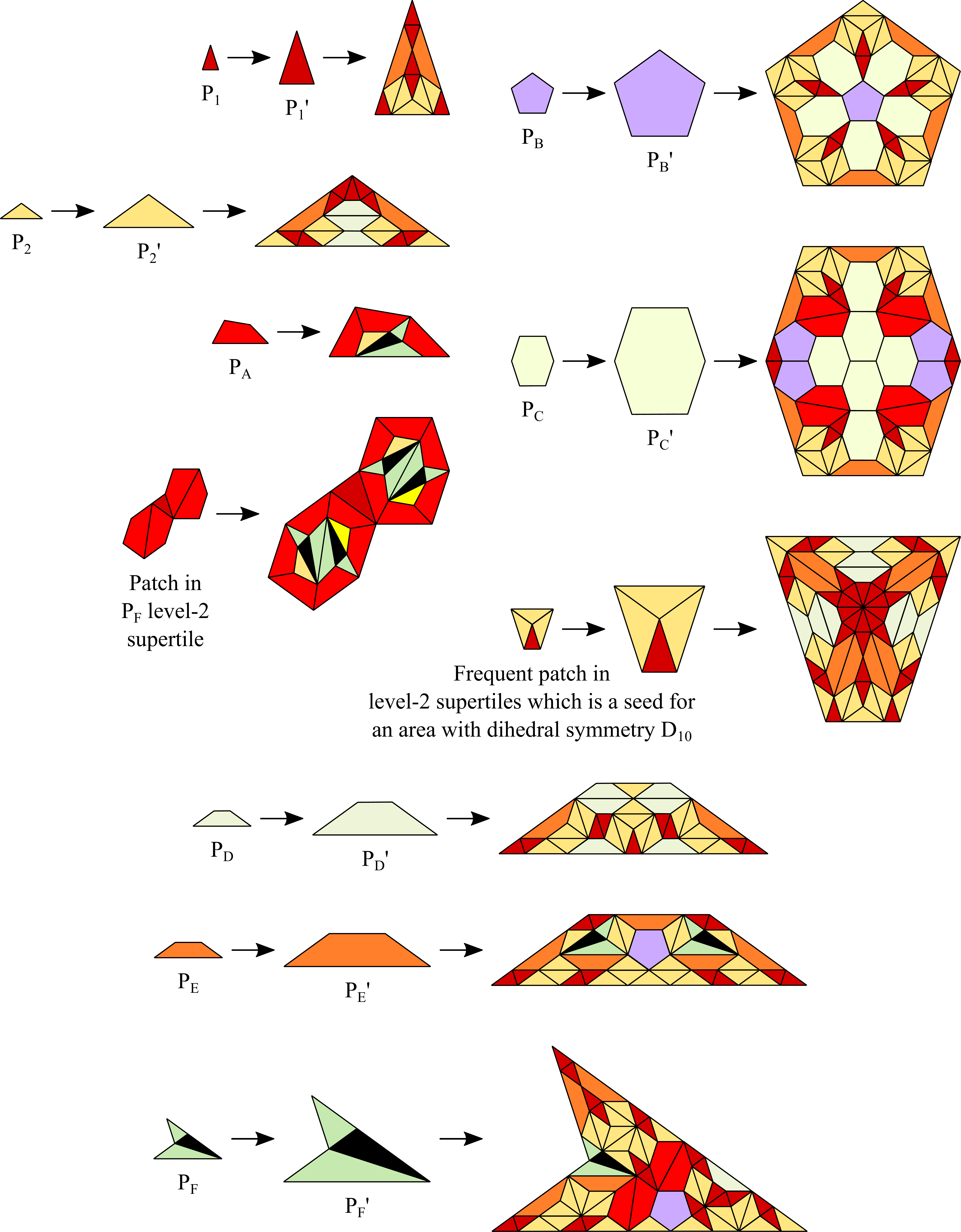}}
\end{center}\caption{\label{fig:CAST_DTO_5}Substitution rules of a CAST with DTO and FLC
for the case $n=5$ with minimal inflation multiplier and individual
dihedral symmetry $D_{10}$. Here the $P_{A}$ level-$3$ supertile
contains two congruent $P_{2}$ tiles (highlighted in yellow) which
are rotated against each other by an irrational angle. }
\end{figure}
\begin{figure}
\begin{center}
\resizebox{1.0\textwidth}{!}{%

\includegraphics{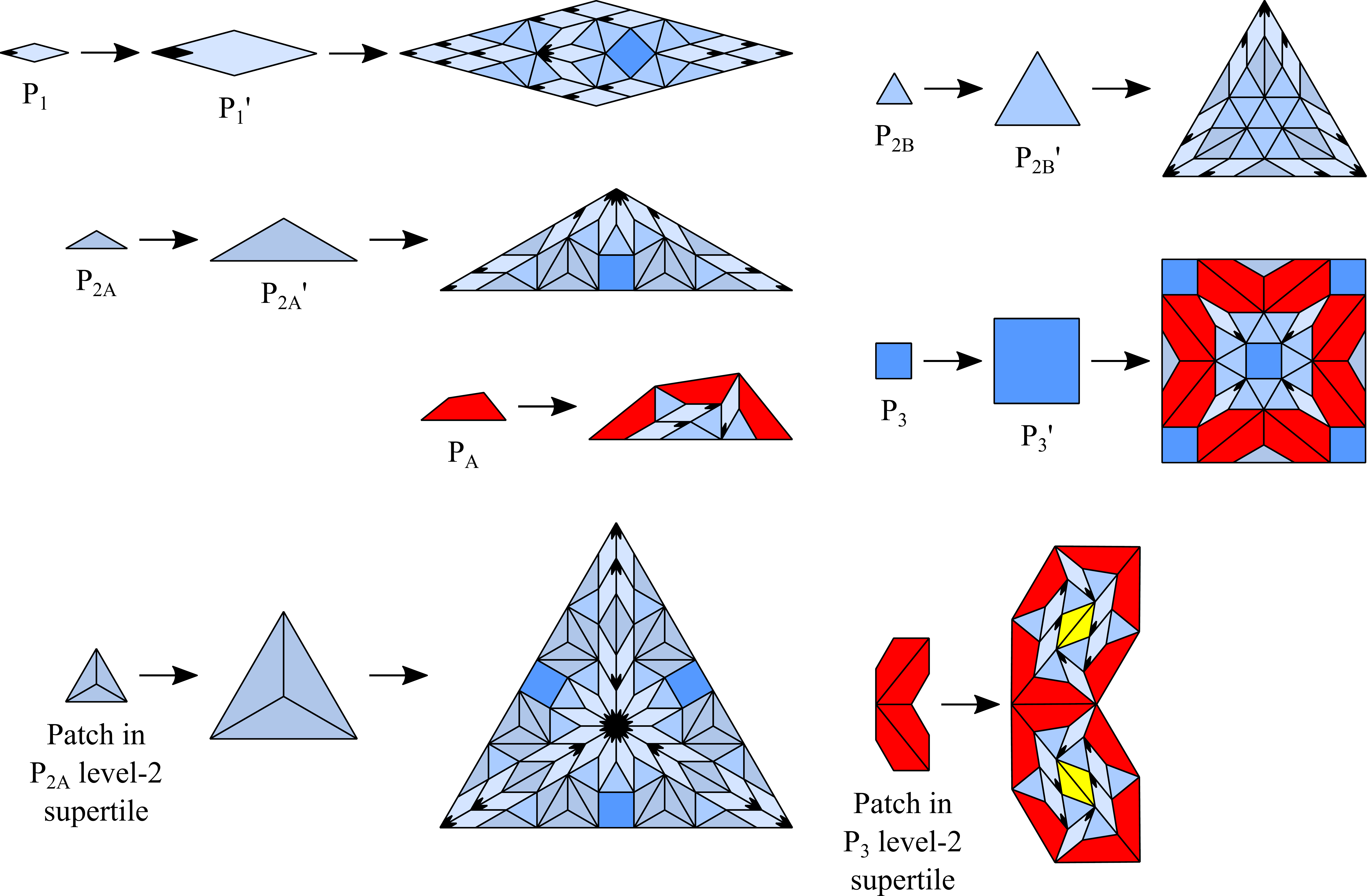}}
\end{center}\caption{\label{fig:CAST_DTO_6}Substitution rules of a CAST with DTO and FLC
for the case $n=6$ with minimal inflation multiplier and individual
dihedral symmetry $D_{12}$. Here the $P_{3}$ level-$3$ supertile
contains two congruent patches (highlighted in yellow) which are rotated
against each other by an irrational angle. }
\end{figure}
\begin{figure}
\begin{center}
\resizebox{1.0\textwidth}{!}{%

\includegraphics{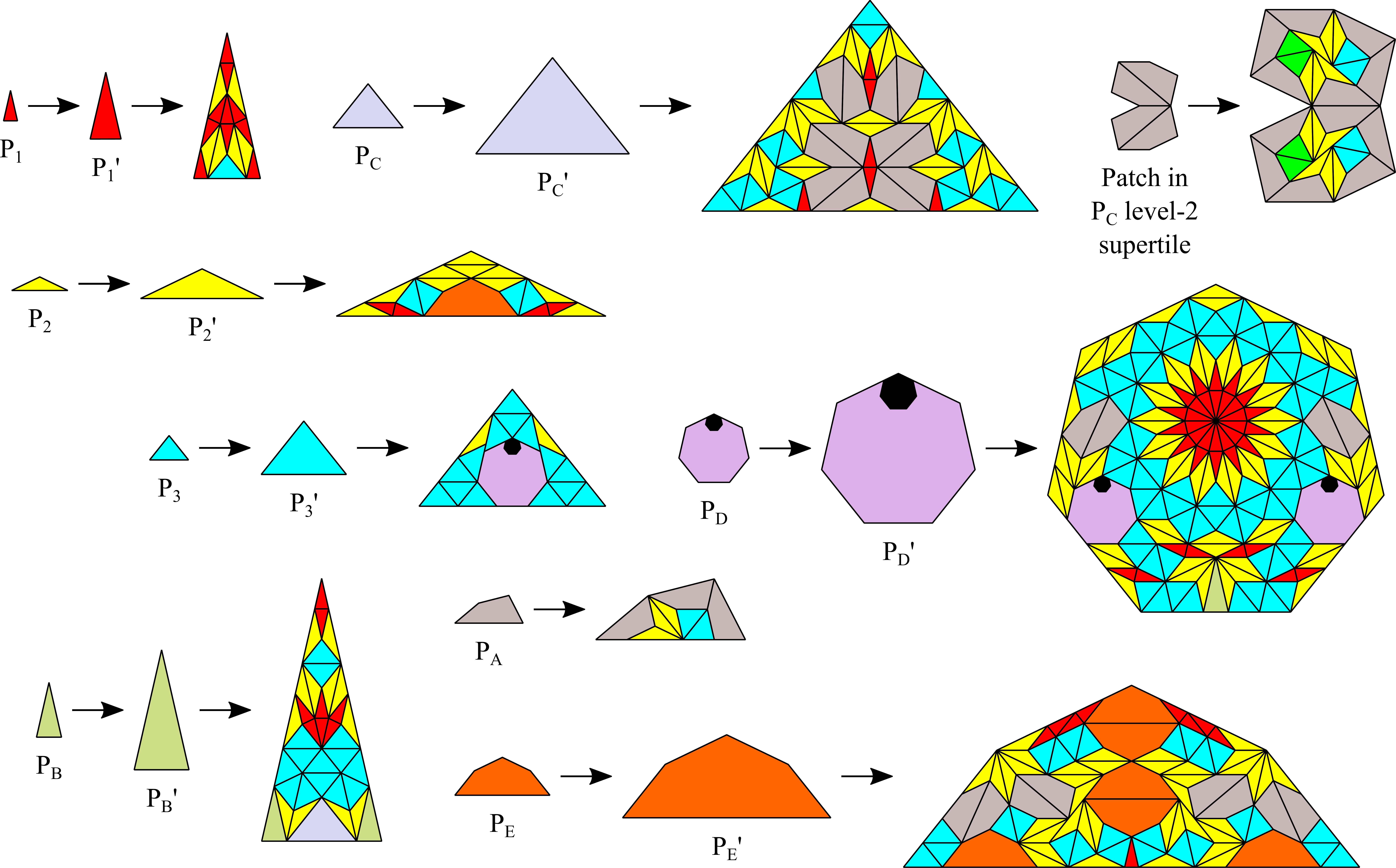}}
\end{center}\caption{\label{fig:CAST_DTO_7}Substitution rules of a CAST with DTO for the
case $n=7$ with minimal inflation multiplier and individual dihedral
symmetry $D_{14}$. Here the $P_{C}$ level-$3$ supertile contains
two congruent patches (highlighted in green) which are rotated against
each other by an irrational angle.}
\end{figure}
\begin{figure}
\begin{center}
\resizebox{0.9\textwidth}{!}{%

\includegraphics[angle=90]{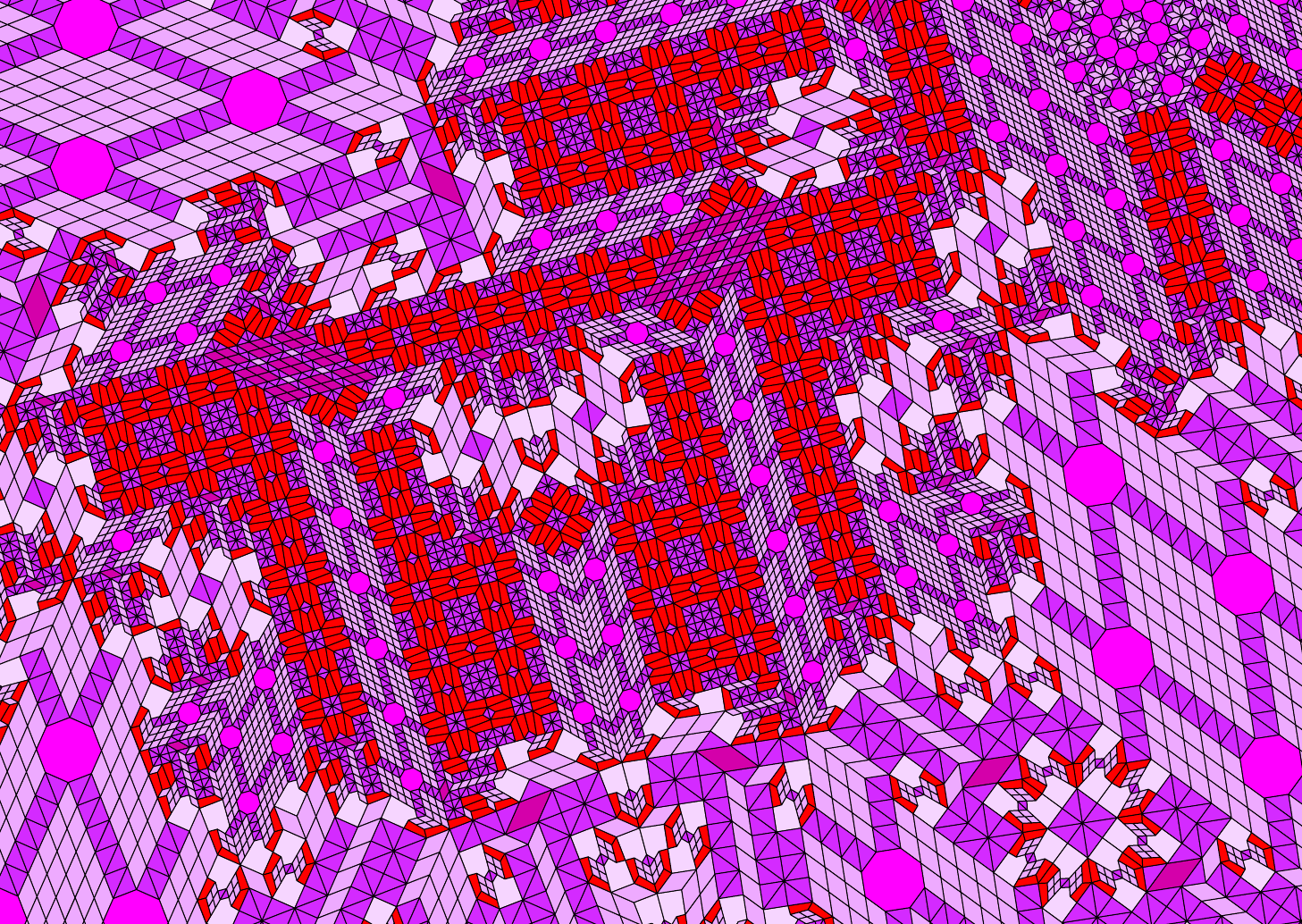}}
\end{center}

\caption{CAST with DTO for the case $n=4$ with minimal inflation multiplier
and individual dihedral symmetry $D_{8}$. It is generated with substitution
rules shown in Figure~\ref{fig:CAST_DTO_4}.}
\end{figure}

As in \citep{FRETTLOH2017120} and \citep{Say-awen:eo5083} additional
vertices (pseudo-vertices) are introduced to dissect all edges of
length $2$ or $3$ into segments of length $1$. These edges can
be found on the boundaries of the $P_{A}$-prototiles of the CASTs
with $n\in\{2,3,4\}$ and the $P_{B}$-prototile of the CAST with
$n=4$ as shown in Figures~\ref{fig:CAST_DTO_2}~-~\ref{fig:CAST_DTO_4}.
Another pseudo-vertex has to be introduced to divide one of the long
edges of the $P_{F}$-prototile of the CAST with $n=5$ into two segments
of length $1$ and $\zeta_{10}^{1}+\overline{\zeta_{10}^{1}}$ as
shown in Figure~\ref{fig:CAST_DTO_5}. Here an appropriate decoration
is applied.

While in \citep{FRETTLOH2017120} and \citep{Say-awen:eo5083} the
derivation of substitution rules starts with a regular $n^{*}$-gon
prototile, we choose the rather irregular convex $P_{A}$-prototile
for this purpose. Its boundary can be described as polygonal line
of unit segments - analogous to roots of unity added up to the minimal
inflation multiplier $\eta_{min.irr.}$ - and a straight ``baseline''
of length $\left|\eta_{min.irr.}\right|$ which connects start and
end point. 

The substitution rule for a $P_{A}$-prototile can easily be obtained
by lining up $P_{A}$-tiles by their baselines - again analogous to
roots of unity added up to the minimal inflation multiplier $\eta_{min.irr.}$
- and filling up the remaining gaps with rhombs, isosceles triangles
or other appropriate polygons. The cases $n\in\{4,5\}$ require the
introduction of asymmetric concave prototiles, in detail the $P_{B}$-prototile
in the shape of a irregular concave hexagon in Figure~\ref{fig:CAST_DTO_4}
and the $P_{F}$-tile in the shape of a concave tetragon (a kite)
in Figure~\ref{fig:CAST_DTO_5}. It may take two inflation steps
applied to the newly introduced prototiles before a dissection into
prototiles is possible and the substitution rule can be found. During
the process additional gaps in new shapes, or more accurately, new
prototiles may appear which also require additional substitution rules.
As a consequence it may take several recursions and a process of trial
and error until all prototiles and substitution rules have been found
and the tiling yields the desired properties such as individual dihedral
symmetry, primitivity, DTO and FLC.

On the strength of the author's experience, the introduction of individual
dihedral symmetry as well as FLC tend to increase the number of prototiles
and substitution rules.

To ensure primitivity it is helpful to place pairs of $P_{A}$-tiles
(with the baseline encased) into substitution rules of other prototiles.
DTO can be ensured by placing two such pairs of $P_{A}$-tiles with
different chiralities into one substitution rule.

It seems the limitation to $n\in\{2,3,4,5,6,7\}$ is a rather arbitrary
choice. For this reason we note:
\begin{conjecture}
\label{conj:CAST_DTO_min_ifl_irr_arg_for_all_n}CASTs with DTO, FLC,
minimal inflation multiplier as shown in Table~\ref{tab:min-infl-with-dto}
and individual dihedral symmetry $D_{2n}$ exist for all $n\geqq2$.
\end{conjecture}
The status of Conjecture~\ref{conj:CAST_DTO_min_ifl_irr_arg_for_all_n}
is subject to further research. The author believes it is true, however
a rigorous proof is not available yet.

\section{Summary}

The definition of Cyclotomic Aperiodic Substitution Tilings (CASTs)
in \citep[Chapter 2]{sym9020019} was revised in a way that it covers
CASTs with dense tile orientations (DTO) and so the results in \citep[Chapter 2]{sym9020019}
apply. It was shown that every CAST with DTO has an inflation multiplier
$\eta$ with irrational argument so that $\nicefrac{k\pi}{2n}\neq\arg\left(\eta\right)\notin\mathbb{\pi Q}$.
We could derive the minimal inflation multiplier for such tilings
for $odd\;n\geqq3$ and $even\;n\geqq2$ with $n\;\mathrm{mod\;}4=2$.
For $even\;n\geqq2$ with $n\mathrm{\;mod\;}4=0$ our results depend
on the status of Conjecture~\ref{conj:no_eta_for_4plusmu2_or_5plusmu2}.
However, for $n\leqq100$ it was confirmed numerically.

For $n\in\left\{ 2,3,4,5,6,7\right\} $ we found examples for CASTs
with DTO, minimal inflation multiplier and individual dihedral symmetry
$D_{2n}$. The examples for $n\in\left\{ 2,3,4,5,6\right\} $ also
yield finite local complexity (FLC). It is very likely that CASTs
with DTO, FLC, minimal inflation multiplier and individual dihedral
symmetry $D_{2n}$ for exist for all $n\geqq2$. 

As a final remark the author would like to emphasise the great aesthetic
qualities of CASTs with DTO - such as their ability to create a subtle
impression of imperfection - which may motivate further research beyond
their mathematical and geometrical properties.

\section*{Acknowledgment}

The author would like to thank Michael Baake (Bielefeld University),
Dirk Frettlöh (Bielefeld University), Uwe Grimm (The Open University,
Milton Keynes), Reinhard Lück and Christian Georg Mayr (Technische
Universität Dresden) for their continued support and encouragement.

\bibliographystyle{halpha}
\bibliography{girih7-reference___}

\begin{thebibliography}{SadlPF18}

\bibitem[BG13]{oro38933}
M.~Baake and U.~Grimm.
\newblock {\em Aperiodic Order. Vol 1. A Mathematical Invitation}, volume 149
  of {\em Encyclopedia of Mathematics and its Applications}.
\newblock Cambridge University Press, Cambridge, 2013.

\bibitem[BG17]{baake2017aperiodic}
M.~Baake and U.~Grimm.
\newblock {\em Aperiodic Order. Vol 2. Crystallography and Almost Periodicity},
  volume 166 of {\em Encyclopedia of Mathematics and its Applications}.
\newblock Cambridge University Press, Cambridge, 2017.

\bibitem[FGH]{HFonl}
D.~Frettl\"{o}h, F.~Gähler, and E.~O. Harris.
\newblock Tilings encyclopedia.
\newblock available at \url{http://tilings.math.uni-bielefeld.de/}.

\bibitem[FR12]{frettloeh_and_richard_2014}
D.~Frettlöh and C.~Richard.
\newblock Dynamical properties of almost repetitive delone sets.
\newblock {\em Discrete and Continuous Dynamical Systems}, 34, 10 2012.

\bibitem[{Fre}08a]{2008PMag...88.2033F}
D.~{Frettlöh}.
\newblock {About substitution tilings with statistical circular symmetry}.
\newblock {\em Philosophical Magazine}, 88:2033--2039, May 2008, 0803.2172.

\bibitem[Fre08b]{FRETTLOH20081881}
D.~Frettlöh.
\newblock Substitution tilings with statistical circular symmetry.
\newblock {\em European Journal of Combinatorics}, 29(8):1881 -- 1893, 2008.

\bibitem[FSdlP17]{FRETTLOH2017120}
D.~{Frettl{\"o}h}, A.~L.~D. {Say-awen}, and M.~L. A.~N. de~las Pe{\~n}as.
\newblock Substitution tilings with dense tile orientations and n-fold
  rotational symmetry.
\newblock {\em Indagationes Mathematicae}, 28(1):120 -- 131, 2017.
\newblock Special Issue on Automatic Sequences, Number Theory, and Aperiodic
  Order.

\bibitem[GS87]{Grunbaum:1986:TP:19304}
B.~Gr\"{u}nbaum and G.~C. Shephard.
\newblock {\em Tilings and Patterns}.
\newblock W. H. Freeman \& Co.: New York, NY, USA, 1987.

\bibitem[Pau17]{sym9020019}
S.~Pautze.
\newblock Cyclotomic aperiodic substitution tilings.
\newblock {\em Symmetry}, 9(2):19, 2017.

\bibitem[Rad94]{10.2307/2118575}
C.~Radin.
\newblock The pinwheel tilings of the plane.
\newblock {\em Annals of Mathematics}, 139(3):661--702, 1994.

\bibitem[Rad99]{radinmiles}
C.~Radin.
\newblock {\em Miles of Tiles}.
\newblock Student mathematical library. American Mathematical Soc., 1999.

\bibitem[Sad98]{sadun98some}
L.~Sadun.
\newblock Some generalizations of the pinwheel tiling.
\newblock {\em Discrete \& Computational Geometry}, 20, 1998.

\bibitem[SadlPF18]{Say-awen:eo5083}
A.~L.~D. Say-awen, M.~L. A.~N. de~las Pe{\~{n}}as, and D.~Frettl{\"{o}}h.
\newblock {Primitive substitution tilings with rotational symmetries}.
\newblock {\em Acta Crystallographica Section A}, 74(4):388--398, Jul 2018.

\end{thebibliography}

\end{document}